\documentclass[reqno,12pt,letterpaper]{amsart}
\usepackage{amsmath,amssymb,amsthm,graphicx,mathrsfs,url}
\usepackage[usenames,dvipsnames]{color}
\usepackage[colorlinks=true,linkcolor=Red,citecolor=Green]{hyperref}

\def\?[#1]{\textbf{[#1]}\marginpar{\Large{\textbf{??}}}}

\newtheorem{theo}{Theorem}
\newtheorem{prop}{Proposition}[section]

\newtheorem{lem}[prop]{Lemma}
\newtheorem{cor}[prop]{Corollary}

\numberwithin{equation}{section}

\DeclareMathOperator{\supp}{supp}
\DeclareMathOperator{\Vol}{dvol}
\DeclareMathOperator{\vol}{Vol}
\DeclareMathOperator{\Tr}{Tr}
\def\plM{\partial\overline M}

\newcommand{\mc}{\mathcal}
\newcommand{\rr}{\mathbb{R}}
\newcommand{\nn}{\mathbb{N}}
\newcommand{\cc}{\mathbb{C}}
\newcommand{\hh}{\mathbb{H}}
\newcommand{\zz}{\mathbb{Z}}

\newcommand{\la}{\lambda}
\newcommand{\eps}{\varepsilon}

\newcommand{\pl}{\partial}
\newcommand{\x}{\times}

\newcommand{\til}{\widetilde}
\newcommand{\bbar}{\overline}

\newcommand{\cjd}{\rangle}
\newcommand{\cjg}{\langle}

\newcommand{\demi}{\frac{1}{2}}
\newcommand{\ndemi}{\frac{n}{2}}
\newcommand{\tra}{\textrm{Tr}}

\newcommand{\indic}{\operatorname{1\hskip-2.75pt\relax l}}

\title[Scattering phase asymptotics with fractal remainders]%
{Scattering phase asymptotics\\
with fractal remainders}
\author{Semyon Dyatlov}
\email{dyatlov@math.berkeley.edu}
\address{Department of Mathematics, Evans Hall, University of California,
Berkeley, CA 94720, USA}
\author{Colin Guillarmou}
\email{cguillar@dma.ens.fr}
\address{DMA, U.M.R. 8553 CNRS, \'Ecole Normale Superieure, 45 rue d'Ulm,
75230 Paris cedex 05, France}

\begin{document}

\begin{abstract}
For a Riemannian manifold $(M,g)$ which is isometric to the Euclidean
space outside of a compact set, and whose trapped set has Liouville
measure zero, we prove Weyl type asymptotics for the scattering phase
with remainder depending on the classical escape rate and the maximal
expansion rate. For Axiom~A geodesic flows, this gives a polynomial
improvement over the known remainders. We also show that the remainder
can be bounded above by the number of resonances in some
neighbourhoods of the real axis, and provide similar asymptotics for
hyperbolic quotients using the Selberg zeta function.
\end{abstract}

\maketitle

\addtocounter{section}{1}
\addcontentsline{toc}{section}{1. Introduction}

In this paper, we derive high energy asymptotics for the
\emph{spectral shift function}, or the \emph{scattering phase}, with
remainders which depend on the dynamic behavior of the underlying
classical system. Our main setting is an $n+1$ dimensional Riemannian
manifold $(M,g)$ which is Euclidean near infinity, but we shall also
consider in the second part of the paper the case of convex co-compact
hyperbolic quotients $\Gamma\backslash \mathbb H^{n+1}$ using a
different approach.  If $\Delta$ is the (nonnegative)
Laplace--Beltrami operator on $M$ and $\Delta_0$ is a reference
operator (such as the Laplacian on the free space when $(M,g)$ is a
metric perturbation of the Euclidean space), then the spectral shift
function is the trace of the difference between the spectral
projectors of $\Delta$ and $\Delta_0$~-- see~\eqref{defspectral}
and~\eqref{defofspect}. The scattering phase is the phase of the
determinant of the relative scattering matrix. In a very general
setting, these two functions are equal almost everywhere on the
absolutely continuous spectrum, as shown in the classical work of
Birman--Krein~\cite{BiKr}, thus we will denote both of them by the
symbol $s(z)$, where $z\in \mathbb R$ is the frequency.

Our recent work~\cite{DyGu} contains a local Weyl law with a remainder
expressed in terms of classical escape rate~-- in particular, if the
metric has constant curvature $-1$ near the trapped set $K$ and the
Hausdorff dimension of $K$ is $\dim_H(K)=2\delta+1$, then the
remainder is $\mathcal O(z^{\delta+})$. If $M$ is Euclidean near
infinity, we use the commutator method of Robert~\cite{Ro} to deduce
an asymptotic expansion of $s(z)$ with the same remainder~-- see
Theorem~\ref{t:trace}. This seems to be the first result on spectral
asymptotics with a fractal remainder which is intermediate between the
usual Weyl law and the non-trapping case where there is a complete
expansion at high frequency. By contrast, the best known remainder for
Weyl law on compact manifolds with chaotic geodesic flows is $\mathcal
O(z^n/\log z)$, see B\'erard~\cite{berard}.

We next give several other estimates on the spectral shift function in
different settings, with remainders related to the one discussed in
the previous paragraph. First of all, the function $s(z)$ admits a
meromorphic extension into the whole complex plane, and the asymptotic
behavior of $s(z)$ for large real $z$ is related to the distribution
of its poles, known as scattering poles or resonances.  In the
Euclidean near infinity setting, we get an asymptotic expansion of the
scattering phase with remainder depending on the number of resonances
in the ball of size $c\log(z)$ centered at $z$, for some $c$ depending
on the injectivity radius~-- see Theorem~\ref{usingresonance}. Next,
for the case of hyperbolic surfaces $\Gamma\backslash \mathbb H^2$, we
get an asymptotic expansion of $s(z)$ using the Selberg zeta
function~-- see Theorem~\ref{t:scattering-phase-special}.

\subsection{Euclidean near infinity case~-- an application of~\cite{DyGu}}

Let $g_0$ and $g_1$ be two smooth metrics on a manifold $M$. When
there exists a compact set $N\subset M$ such that each $(M\setminus
N,g_j)$ is isometric to $\rr^{n+1}\setminus B(0,R_0)$ for some
$R_0>0$, we define (up to a constant) the spectral shift function
$s(z)$ for the pair $(\Delta_{g_1},\Delta_{g_0})$ by duality: for all
$\varphi\in C_0^\infty(\rr)$,
\begin{equation}\label{defspectral}
\int_0^\infty\varphi(z) \partial_z s(z)\,dz
=\Tr(\varphi(\sqrt{\Delta_{g_0}})-\varphi(\sqrt{\Delta_{g_1}}))
\end{equation}
where it is straightforward to see that the trace is well defined,
using that~$g_0=g_1$ outside a compact set.  In this setting, $s(z)$
will be a smooth function on $(0,\infty)$ and it is equal, up to a multiplicative constant, to
the phase of the relative scattering operator by the work of
Birman--Krein~\cite{BiKr}. 

We shall also be interested in the case where $(M_0,g_0)$ and
$(M_1,g_1)$ are two Riemannian manifolds such that for each $j=0,1$
there is a compact set $N_j\subset M_j$ with $M_j\setminus N_j$
isometric to $\rr^{n+1}\setminus B(0,R_0)$. The spectral shift
function for the pair $(\Delta_{g_1},\Delta_{g_0})$ can then be
defined using the black-box trace~\cite{Sj} by the formula
\begin{equation}\label{defofspect}
\int_0^\infty\varphi(z) \partial_z s(z)\,dz
=\tra_{\rm bb}(\varphi(\sqrt{\Delta_{g_0}})-\varphi(\sqrt{\Delta_{g_1}})),
\end{equation}
where the trace $\tra_{\rm bb}$ is defined in~\eqref{bbtrace}.

Before we state our result, we need to define a few geometric
quantities. Let $(M,g)$ be a Riemannian manifold such that, for some
compact set $N\subset M$, the end $(M\setminus N,g) $ is isometric to
$\rr^{n+1}\setminus B(0,R_0)$ for some $R_0>0$.  Let $g^t$ be the
geodesic flow acting on the unit cotangent bundle $S^*M$. The
\emph{trapped set} $K\subset S^*M$ is defined as follows: $(m,\nu)$
lies in $K$ if and only if the corresponding geodesic $g^t(m,\nu)$
lies entirely in some compact subset of $S^*M$. Denote by $\mu_L$ the
Liouville measure on $S^*M$ generated by the function
$\sqrt{p(m,\nu)}=|\nu|_g$.  The set $\mathcal T(t)$ of geodesics
trapped for time $t>0$ is defined as
\begin{equation}
  \label{e:T-t}
\mathcal T(t):=\{(m,\nu)\in S^*M\mid m\in N,\ \pi(g^t(m,\nu))\in N\}.
\end{equation}
where $\pi :S^*M\to M$ is the canonical projection.
We define the \emph{maximal expansion rate} of the geodesic flow as follows:
\begin{equation}
  \label{lamax}
\Lambda_{\max}:=\limsup_{|t|\to +\infty}{1\over |t|}\log \sup_{(m,\nu)\in \mathcal T(t)}
\|dg^t(m,\nu)\|.
\end{equation}  
The norm on the right hand side is with respect to the Sasaki metric.  
Using our recent work~\cite{DyGu},  we obtain
%
%
\begin{theo}
  \label{t:trace}
Let $(M,g)$ be a smooth Riemannian manifold for which there is a
compact subset $N$ such that $M\setminus N$ is isometric to
$\rr^{n+1}\setminus B(0,R_0)$. Assume that the trapped set $K$ has
Liouville measure $0$ and let $s(z)$ be the spectral shift function
associated to the pair $(\Delta_g, \Delta_{\rr^{n+1}})$ defined
by~\eqref{defofspect}. Let $\Lambda_{\max}$ be defined
in~\eqref{lamax} and $\Lambda_0>\Lambda_{\max}$,Ê then there exist
some coefficients $c_j$ such that as $z\to \infty$ we have for all
$L\in \nn$
\begin{equation}
  \label{e:s-estimate}
s(z)=\sum_{j=0}^Lc_jz^{n+1-j}+\mathcal O\big(z^n\mu_L\big(
\mathcal T(\Lambda_0^{-1}\log|z|)\big)\big)+\mathcal O(|z|^{n-L}).
\end{equation}
\end{theo}
%
%
The coefficients $c_j$ are integrals of local Riemannian invariants,
they appear in the small time asymptotics for the local trace of the
heat kernel or Schr\"odinger propagator (e.g. see Robert \cite{Ro});
for instance, $c_0=C_n({\rm Vol}_{\rm eucl}(B(0,R_0))-{\rm Vol}_g(N))$
where $C_n$ is a universal constant depending only on $n$. When
$K=\emptyset$, we recover a full expansion for $s(z)$, when
$K\not=\emptyset$, the remainder is not better than $\mc{O}(1)$.

In the proof, we need to assume that $M$ is Euclidean at infinity for
\eqref{e:s-estimate} since we use the commutator method of Robert
\cite{Ro} to express the scattering phase as a trace of the spectral
projector in a compact region, and then we apply our estimate
\cite[Theorem~3]{DyGu} on the expansion as $h\to 0$ of the trace
$\tra(A\indic_{[0,\la]}(h^2\Delta_g))$ when $\la>0$ is in a compact
interval and $A\in \Psi^0(M)$ is any compactly supported semiclassical
pseudodifferential operator. It is likely that the result extends to
perturbations of $\rr^{n+1}$ which are not compactly supported.\\

We can now specialize to particular cases where the flow is
\emph{uniformly partially hyperbolic} in the following sense: there
exists $\la>0$ and an invariant splitting of $TS^*M$ over $K$ into
continuous subbundles
\[
T_{z}S^*M=E^{cs}_z\oplus E^u_z, \quad \forall z\in K
\]
such that the dimensions of $E^u$ and $E^{cs}$ are constant on $K$ and for all $\eps>0$, there is $t_0\in \rr$ such that 
\[
\forall z\in K, \,\, \forall t\geq t_0,  
\left\{\begin{array}{ll}
\forall v\in E^{u}_z,\,\, \|dg_z^t v\|\geq e^{\la t}\|v\|,\\
\forall v\in E^{cs}_z, \,\, \|dg^t_zv\|\leq  e^{\eps t}\|v\|.
\end{array}\right.
\]
We can define $J^u$, the unstable Jacobian of the flow, by  
\[
J^u(z):=-\pl_{t}(\det dg^t(z)|_{E^u_z})|_{t=0}
\]
where $dg^t: E^u_z\to E^u_{g^t(z)}$ and the determinant is defined
using the Sasaki metric (to choose orthonormal bases in $E^u$). The
\emph{topological pressure} of a continuous function $\varphi:K\to
\rr$ with respect to the flow can be defined by the variational
formula
\begin{equation}\label{topopressure}
P(\varphi):=\sup_{\mu\in \mc{M}(K)}\Big(h_{\mu}(g^1)+\int \varphi d\mu\Big)
\end{equation}
where $\mc{M}(K)$ is the set of $g^t$-invariant Borel probability
measures and $h_\mu(g^1)$ is the measure theoretic entropy of the flow
at time $1$ with respect to $\mu$. Young \cite{Yo} proved that for
uniformly partially hyperbolic flows, the classical escape rate
$\lim_{t\to +\infty}\frac{1}{t}\log(\mu_L(\mathcal T(t))$ is equal to
$P(J^u)$. Combining this with Theorem~\ref{t:trace}, we obtain
%
%
\begin{cor}
  \label{uph}
Let $(M,g)$ be a Riemannian manifold satisfying the assumptions of
Theorem \ref{t:trace}, and $\Lambda_0>\Lambda_{\max}$. If the geodesic
flow is uniformly partially hyperbolic on $K\not=\emptyset$ with
negative topological pressure $P(J^u)<0$, then
\begin{equation}\label{axiomAcase}
s(z)=\sum_{j=0}^{n+1} c_jz^{n+1-j}+\mathcal O(z^{n+P(J^u)/\Lambda_0}).
\end{equation}
In particular, if $g$ has curvature $-1$ near the trapped set $K$
and the latter has Hausdorff dimension $2\delta+1$, then
\begin{equation}\label{constcurv}
 s(z)=\sum_{j=0}^{[n+1-\delta]} c_jz^{n+1-j}+\mathcal O(z^{\delta+}).
 \end{equation}
\end{cor}
%
%
When the flow is Axiom~A, one has $P(J^u)<0$ by
Bowen--Ruelle~\cite[Theorem~5]{BoRu}. We refer the reader
to~\cite[Appendix~B]{DyGu} for more discussions about the classical
escape rate and the Hausdorff dimension of the trapped set.
 
\subsection{Relation with resonances}

We also compare these asymptotics to the counting function of
resonances for the Laplacian $\Delta_g$ in regions close to the
continuous spectrum. By definition, resonances are poles of the
meromorphic continuation of the resolvent $R(z)=(\Delta_g-z^2)^{-1}$
from ${\rm Im}(z)<0$ to $z\in \cc$ as an operator mapping $L^2_{\rm
comp}(M)$ to $L^2_{\rm loc}(M)$. We show
%
%
\begin{theo}
  \label{usingresonance}
Let $(M,g)$ be a smooth Riemannian manifold which is isometric to
$\rr^{n+1}\setminus B(0,R_0)$ outside a compact set $N$, with $n+1$ odd,
and let $\mc{R}$ be the set of resonances of $\Delta_g$ in $\{{\rm
Im}(z)\geq 0\}$. If $t_0$ is the radius of injectivity of $g$, and if
for some $c>0$ satisfying $ct_0>n+1$ the following estimate holds for
$z>0$ large
\begin{equation}
  \label{e:required-fwl}
\sharp\{\rho \in \mc{R};  |z-\rho|\leq c\log(z)\}=\mc{O}(z^{\alpha}),
\end{equation}
for some $\alpha\in [0,n]$, then the spectral shift function defined
by \eqref{defofspect} satisfies the asymptotics as $z\to \infty$
\[
s(z)=\sum_{j=0}^{n+1} c_jz^{n+1-j}+\mathcal{O}(z^{\alpha}).
\]
\end{theo}
%
%
The proof uses some ideas of Melrose~\cite{Me1} by expressing the
derivative of $s(z)$ in terms of resonances through a Poisson formula.
Dimassi~\cite{Di} showed that $s(z)$ has a full expansion if there are
no resonances in a logarithmic neighbourhood of the real axis; our
result is (in a way) a generalization of~\cite{Di}. 

Sj\"ostrand--Zworski~\cite{SjZw1} proved that if the geodesic flow is
hyperbolic on $K$, then
\begin{equation}
  \label{e:fwl}
\sharp\{\rho\in \mc{R}; |z-\rho|\leq c\}=\mc{O}(z^{m})
\end{equation}
for all $m$ satisfying $2m+1>\dim_M(K)$, where $\dim_M$ denotes the
Minkowski dimension (in case of pure dimensions, $m$ can be taken to
be $(\dim_M(K)-1)/2$). For semiclassical operators with analytic
coefficients satisfying hyperbolicity assumptions,
Sj\"ostrand~\cite[Theorem~4.6]{Sj0} gives estimates for numbers of
resonances in boxes of real part of size $z^{1/2}$ and imaginary part
of any size between $1$ and $o(z)$; it is feasible that in combination
with the second microlocalization analysis of~\cite{SjZw1} to handle
the real parts, this method gives in particular the
estimate~\eqref{e:required-fwl} with $\alpha$ any number such that
$2\alpha+1>\dim_M(K)$. It is not clear whether an estimate of the
type~\eqref{e:fwl} for resonances in the logarithmic
regions~\eqref{e:required-fwl} holds for general $C^\infty$ metrics.

\subsection{Hyperbolic quotients}

In the hyperbolic near infinity setting, there is no natural model
operator to define a spectral shift function, however in even
dimension one can use the construction of the second
author~\cite{GuAJM}. In dimension $2$, Guillop\'e--Zworski~\cite{GZ2}
defined also a relative scattering phase $s(z)$ by using the Laplacian
on the funnels with Dirichlet condition on a curve separating the
funnels and a compact region of $M$. In constant curvature
$\Gamma\backslash\hh^2$, this curve can be taken to be the boundary of
the convex core, made of disjoint closed geodesics, and $s(z)$ is
related to the argument of the Selberg zeta function
$Z_\Gamma(1/2+iz)$ of $\Gamma$ for $z>0$. Using $Z_\Gamma(\la)$, the
estimate on its growth in strips proved by
Guillop\'e--Lin--Zworski~\cite{GLZ}, and some classical results of
holomorphic functions, we can compare our microlocal result to see that,
at least without more dynamical assumptions, our estimate seems almost sharp:
%
%
\begin{theo}
  \label{t:scattering-phase-special}
Let $M=\Gamma\backslash\hh^2$ be a convex co-compact hyperbolic
surface and let $s(z)$ be the scattering phase defined in~\cite{GZ2}
using the Laplacian on the funnels with Dirichlet condition at the
boundary of the convex core. Let $\delta\in(0,1)$ be the Hausdorff
dimension of the limit set of $\Gamma$, or equivalently $2\delta+1$ is
the Hausdorff dimension of the trapped set $K$. Then:\\
1) $s(z)$ satisfies the asymptotics as $z\to +\infty$
\[
s(z)=\frac{1}{4\pi} {\rm Vol}(N)z^2 +R(z), \quad |R(z)|=\mc{O}(z^\delta),
\quad \Big|\int_{0}^zR(t)dt\Big|=\mc{O}(z^\delta)
\] 
where $N$ is the convex core of $M$.\\
2) the Breit--Wigner formula holds: for 
all $z\in [T-\sigma,T+\sigma]$ with $T$ large and $\sigma$ fixed 
\[
\pl_zs(z)=\frac{1}{4\pi}\vol(N)z^2+\frac{1}{\pi}\sum_{s\in \mc{R}_{\sigma+1}}
\frac{{\rm Im}(\rho)}{|z-\rho|^2}+\mc{O}(T^\delta)
\]
where $\mc{R}_{\sigma+1}$ denotes the set of resonances in the disc of radius $\sigma+1$ centered at $T$.\\
3) If $\Gamma$ is an infinite index sugroup of an arithmetic group
derived from a quaternion algebra, and if $\delta>3/4$, then for all
$\eps>0$, $R(z)$ is {\bf not} an $\mc{O}(z^{2\delta-3/2-\eps})$.
\end{theo}
%
%
The statement 3) is essentially restating a result proved by
Jakobson--Naud~\cite{JaNa}. A result similar to Theorem
\ref{t:scattering-phase-special} (except statement 3)) is also true
for even dimensional convex co-compact hyperbolic Schottky manifolds,
see Theorem~\ref{resultn+1} in Section~\ref{s:krein}.

When $\delta<1/2$ in dimension~$2$, there is actually an exact formula
for $s(z)$ as a converging infinite sum over closed geodesics,
see~\eqref{whendelta<n/2}.

\subsection{Previous results}
 
The asymptotics for the scattering phase has a long history, we will
only give references related to the settings we use in this paper.
The first result giving an asymptotics with the leading term was
announced in Buslaev~\cite{Bu} for a perturbation of $\rr^{d}$ by a
non-trapping obstacle, this was shown with non-sharp remainder in
Majda--Ralston~\cite{MaRa} for the strictly convex case, and by
Jensen--Kato~\cite{JeKa} for star-shaped domains; an asymptotics with
remainder $\mc{O}(z^{d-3})$ was proved for non-trapping compact metric
perturbations in Majda--Ralston \cite{MaRa2}. Then, Petkov--Popov
\cite{PePo} gave a full asymptotic expansion in powers of $z$ for
non-trapping obstacle perturbations of $\rr^{d}$ and a sharp estimate
with $\mc{O}(z^{d-1})$ remainder was shown by Melrose~\cite{Me1} for
general obstacles $\Omega$ in odd dimensions using the counting
function of resonances in large disks of $\cc$. When the set of
transversally reflected geodesics is of measure $0$,
\cite{Me1} obtained a second term of the form ${\rm
Vol}(\pl\Omega)z^{d-1}$ with remainder $o(z^{d-1})$ using a
Duistermaat--Guillemin/Ivrii type argument.

Robert~\cite{Ro} proved a very general result for (not necessarily
compactly supported) second order perturbations of the Euclidean
Laplacian on $\rr^{d}$ with a remainder $\mc{O}(z^{d-1})$ and with a
complete expansion in powers of $t$ for non-trapping perturbations.
In the black-box compact perturbation setting, Christiansen~\cite{Chr}
gave an estimate with a sharp remainder. In the semiclassical setting,
the work of Bruneau--Petkov~\cite{BrPe} extends the results
of~\cite{Ro,Chr}. The Breit--Wigner formula relating resonances and
scattering phase in the semi-classical setting appears in particular
in~\cite{BrPe} and in Petkov--Zworski~\cite{PeZw}. In the setting of
hyperbolic surfaces with funnels, Guillop\'e--Zworski~\cite{GZ2}
proved a Weyl asymptotics with remainder $\mc{O}(z)$ (in fact they
dealt with $n_c$ cusps as well, in which case there is a second term
$-\frac{n_c}{\pi}z\log(z)$ before $\mc{O}(z)$). For convex co-compact
hyperbolic manifolds $\Gamma\backslash \hh^{n+1}$, the second
author~\cite{GuAJM} gave an asymptotics with remainder $\mc{O}(z^{n})$
for the Krein spectral function, and Borthwick~\cite{Bo} extended it
to compact perturbations.

\section{Asymptotics of the spectral shift function for compact perturbations of the Euclidean Laplacian}
\label{s:scattering-phase}

\subsection{Asymptotics using the trace estimate of~\cite{DyGu}}

Let $(M,g)$ be a complete Riemannian manifold satisfying the
assumptions of Theorem~\ref{t:trace}. We put $(M_1,g_1)=(M,g)$ and
$(M_0,g_0)$ to be $\mathbb R^{n+1}$ with the Euclidean metric; let
$N_1=N\subset M_1$ and $N_0=B(0,R_0)\subset M_0$, so that
$(M_1\setminus N_1,g_1)$ is isometric to $(M_0\setminus N_0,g_0)$.  We
define $s(z)$ by duality like in \eqref{defofspect} using the
black-box trace: with the notation $[A_j]_{1}^0=A_0-A_1$, the
black-box trace is defined by
\begin{equation}
  \label{bbtrace}
\begin{split}
\tra_{\rm bb}(\varphi(\sqrt{\Delta_{g_0}})-\varphi(\sqrt{\Delta_{g_1}}))
:=&\big[\Tr(\chi_j\varphi(\sqrt{\Delta_{g_j}})\chi_j )\big]_{1}^0+\big[\Tr((1-\chi_j) \varphi(\sqrt{\Delta_{g_j}})\chi_j)\big]_{1}^0\\
&+\big[\Tr(\chi_j \varphi(\sqrt{\Delta_{g_j}})(1-\chi_j))\big]_{1}^0\\
&+\tra(\big[(1-\chi_j)\varphi(\sqrt{\Delta_{g_j}})(1-\chi_j)\big]_{1}^0).
\end{split}
\end{equation}
Here $\chi_j\in C_0^\infty(M_j)$ is a function which equals $1$ near
$N_j$, and $(1-\chi_1)=(1-\chi_2) \in C^\infty(M_j\setminus N_j)$. All
the traces exist by easy arguments, see~\cite{Sj}, and in the last
trace, we use the identification~$M_1\setminus N_1\sim M_0\setminus
N_0$. The definition agrees with \eqref{defspectral} when $M_1=M_2$.
In this setting, $s(z)$ is always an analytic function on
$[0,\infty)$, this can be seen for instance by using the Birman--Krein
formula relating $s(z)$ with the argument of the determinant of the
relative scattering operator (which is analytic and of modulus $1$ on
the real line).\\

Let us first recall Theorem~4 in~\cite[Section~5.3]{DyGu}. We refer to
\cite{d-s,e-z} or \cite[Section~3]{DyGu} for the definition of
semiclassical pseudodifferential operators $\Psi^{*}(M)$ and
semiclassical quantizations ${\rm Op}_h$.
%
%
\begin{theo}\label{asympofs_A}
Let $(M,g)$ be a complete Riemannian manifold satisfying the
assumptions of Theorem \ref{t:trace} and $\Lambda_0>\Lambda_{\max}$
with $\Lambda_{\max}$ defined in \eqref{lamax}. Let $P(h)=h^2\Delta_g$
and let $A={\rm Op}_h(a)\in \Psi^0(M)$ be a compactly supported
semi-classical pseudodifferential operator. Then there exist some
differential operators $L_j$ of order $2j$ on $T^*M$, depending on the
quantization procedure ${\rm Op}_h$, with $L_0=1$, such that for all
$L\in\nn$, all $\la>0$, we have as $h\to 0$
\[\begin{split}
{\rm Tr}(A\indic_{[0,\la]}(P(h))= & (2\pi h)^{-n-1}\sum_{j=0}^L h^j
\int\limits_{|\nu|_g^2\leq \la}L_ja\, d\mu_\omega+h^{-n}\mc{O}\big(\mu_L(\mc{T}(\Lambda_0^{-1}|\log h|)+h^{L}\big)
\end{split}\]
where the remainder is uniform for $\la$ in a compact interval of $(0,\infty)$
and $\mu_\omega$ is the symplectic volume form on $T^*M$.
\end{theo}
%
%
Using Theorem~\ref{asympofs_A} and the commutator method of Robert \cite{Ro}, we will show   
Theorem~\ref{t:trace}.
%
%
\begin{proof}[Proof of Theorem~\ref{t:trace}]
Define $(M_1,g_1)$ and $(M_0,g_0)$ as in the beginning of this section
and put $\mathcal E=M_1\setminus N_1$, which is identified with
$B^c:=M_0\setminus N_0$.  Let $\chi_j\in C_0^\infty(M_j)$ be the
cutoff functions used in the definition~\eqref{bbtrace} of the black
box trace.  Let $\mc{A}$ be a smooth first order differential operator
supported in $\mc{E}$, which is equal to
$\mc{A}=\frac{1}{4}(m.\nabla+\nabla.m)$ on the support of $1-\chi_1$,
where $m\in \rr^{n+1}$ is the Euclidean coordinate in $\mc{E}$. A
computation shows that $[\Delta_{g_0},\mc{A}]=\Delta_{g_0}$ in $B^c$.
Therefore, since $\Delta_{g_1}=\Delta_{g_0}$ on $\mc{E}$, we get,
using the notation $[A_j]_{1}^0=A_0-A_1$,
\[
\begin{split}
\int_0^\infty\varphi(z^2) z^2\partial_z s(z)\,dz=& 
[\tra(\chi _j\Delta_{g_j}\varphi(\Delta_{g_j}))]_{1}^0+[\tra((1-\chi _j)\Delta_{g_j}\varphi(\Delta_{g_j})\chi_j)]_{1}^0\\
& +\tra([(1-\chi_j)[\Delta_{g_j},\mc{A}]\varphi(\Delta_{g_j})(1-\chi_j)]_{1}^0) 
\end{split}
\]
for all $\varphi\in C_0^\infty((0,\infty))$, and both traces make
sense. Using the argument of Robert~\cite[Appendice~A]{Ro}, it is
direct to see by a limiting argument (using cutoff functions supported
on large balls) and the cyclicity or the trace that
\[
\begin{split} 
\tra([(1-\chi_j)[\Delta_{g_j},\mc{A}]\varphi(\Delta_{g_j})(1-\chi_j)]_1^0)= 
\tra(Q_0\varphi(\Delta_{g_0}))
-\tra(Q_1\varphi(\Delta_{g_1}))
\end{split}
\]
with $Q_j:=((1-\chi_j)[\Delta_{g_j},\chi_j]+[\Delta_{g_j},\chi_j](1-\chi_j))\mc{A}$.
Now we deduce that
\[
\int_0^\infty\varphi(z^2) z^2\partial_z s(z)\,dz=\tra(S_0\varphi(\Delta_{g_0}))-\tra(S_1\varphi(\Delta_{g_1}))
\]
with  
\[
S_j:=Q_j+\chi _j\Delta_{g_j}+\chi_j(1-\chi _j)\Delta_{g_j}
\]
compactly supported second order differential operators. Let $\til{\chi}_j\in C_0^{\infty}(M_j)$ such that 
$\til{\chi}_j=1$ on $\supp \chi_j$, and in particular $\til{\chi}_jS_j=S_j$.
We then obtain for $z>0$
\[
\pl_zs(z)dz=z^{-2}(\tra(S_0d\Pi_{z^2}(\Delta_{g_0})\til{\chi}_0)-\tra(S_1d\Pi_{z^2}(\Delta_{g_1})\til{\chi}_1)
\]
where $d\Pi_{z}(\Delta_{g_j})$ is the spectral measure of
$\Delta_{g_j}$, i.e. such that $\varphi(\Delta_{g_j})=\int_0^\infty
\varphi(z)d\Pi_z(\Delta_{g_j})$ for all $\varphi\in
C_0^\infty(\rr)$. The traces here make sense since
$d\Pi_{z^2}(\Delta_{g_j})$ are operators with smooth kernels for $z>0$
(see e.g.~\cite[Section~6.2]{DyGu}) and $S_j$ and $\til\chi_j$ are
compactly supported.  Let $\eps>0$ be small and let $\psi\in
C_0^\infty((1-\eps,1+\eps))$ equal to $1$ on $[1-\eps/2,1+\eps/2]$,
and let $\varphi\in C_0^\infty((-1,1-\eps/2))$ be such that
$\psi(\la)+\varphi(\la)=1$ for $0\leq \la\leq 1$. Then for $h:=1/z$
small
\[
\begin{split}s(1/h)=&\int_0^{1/h}\psi(h^2u^2)\pl_u s(u)\,du
+\tra_{\rm bb}(\varphi(h^2\Delta_{g_0})-\varphi(h^2\Delta_{g_1}))
\\
= & [\tra(T_j\indic_{[0,1]}(h^2\Delta_{g_j})\til{\chi}_j)]_{1}^0+ \tra_{\rm bb}(\varphi(h^2\Delta_{g_0})-\varphi(h^2\Delta_{g_1}))
\end{split}
\]
where $T_j:=h^2S_j\til{\psi}(h^2\Delta_{g_j})$ with
$\til{\psi}(\la)=\psi(\la)/\la$ in $C_0^\infty((1-\eps,1+\eps))$. In
all terms above, one can remove the functions $\til{\chi}_j$ by
cyclicity of the trace and using $\til{\chi}_jS_j=S_j$.  Observe that
$T_j\in \Psi^{-\infty}(M)$ by using the Helffer-Sj\"ostrand functional
calculus~\cite[Chapters~8-9]{d-s}.  Moreover, since $S_j$ is a
compactly supported differential operator, $T_j$ can be decomposed
into the sum of a compactly supported semi-classical
pseudodifferential operator in $\Psi^{-\infty}(M)$ and of an operator
$W_j\in h^\infty\Psi^{-\infty}(M)$ with Schwartz kernel
$K_{W_j}(m,m')$ compactly supported in $m$ and decaying to infinite
order as $m'\to \infty$. Indeed, an operator in $\Psi^{-\infty}(M)$
has a smooth Schwartz kernel which is an
$\mc{O}((\frac{h}{d(m,m')})^{\infty})$ where $d(m,m')$ is the
Riemannian distance (this comes from integration by parts in the
oscillating integral defining the kernel). We then have
\[
|\tra(W_j\indic_{[0,1]}(h^2\Delta_{g_j}))|\leq \| W_j\|_{\tra}=\mc{O}(h^\infty)
\] 
where $\|\cdot\|_{\tra}$ is the trace norm.  We can then apply
Theorem~\ref{asympofs_A} with the compactly supported operator
$T_j-W_j\in \Psi^{-\infty}(M)$ instead of $T_j$ and we get an
expansion in powers of $h$ for
$\tra(T_1\indic_{[0,1]}(h^2\Delta_{g_1}))$ up to
$\mc{O}(h^{-n}\mu_L(\mc{T}(\Lambda_0^{-1}|\log h|)))$, and a complete
expansion for $\tra(T_0\indic_{[0,1]}(h^2\Delta_{g_0}))$ since $g_0$
has no trapped set.  The second term with $\varphi(h^2\Delta_{g_j})$
has a full expansion by~\cite[Th\'eor\`eme~2.1]{Ro}, this is a
consequence of the functional calculus for $h^2\Delta_{g_j}$, either
using the Helffer--Sj\"ostrand method~\cite[Chapters~8--9]{d-s} or the
Helffer--Robert~\cite{HeRo} approach. This achieves the proof.
\end{proof}
%
%
In fact, this argument would apply similarly to any situation where
the free Laplacian $\Delta_{g_0}$ on $\rr^{n+1}$ is replaced by a
self-adjoint operator $P_0$ satisfying the black-box
setting~\cite{Sj}, with $P_0=\Delta_{g_1}$ on functions supported
outside the black-box and for which there are expansions of
$\tra(A\Pi_{[0,z]}(P_0))$ in powers of $z$ up to
$\mc{O}(z^{n}\mu_L(\mc{T}(\Lambda_0^{-1}\log z)))$ when $A$ is a
compactly supported pseudo-differential operator.

\subsection{Asymptotics using resonances near the real axis}
\label{s:resonances}

For compact perturbations $P$ of the Euclidean Laplacian (this can be
metric, obstacle or potential) in odd dimension $n+1$, the resolvent
$R(z)=(P-z^2)^{-1}$ of $P$ extends meromorphically with poles of
finite rank from $\{{\rm Im}(z)<0\}$ to $\cc$, the poles are called
\emph{resonances} and form a discrete set $\mc{R}\subset \{{\rm
Im}(z)\geq 0\}$. We associate to each resonance $z_j$ its multiplicity
$m_j\in\nn$ as in~\cite{SjZw}. The following sharp upper bounds on the counting function of
resonances for metric perturbations have been proved by
Sj\"ostrand--Zworski~\cite{SjZw} and Vodev~\cite{Vo}
\begin{equation}
  \label{upperbound}
\sharp\{\rho \in \mc{R}; |\rho|\leq R\}=\mc{O}(R^{n+1}).
\end{equation}
%
%
\begin{proof}[Proof of Theorem~\ref{usingresonance}]
Let $\chi\in C_0^\infty(M)$ be a function which is equal to $1$ near
$N$, and define the tempered distribution on $\rr$
\[
u(t):=2\tra(\cos (t\sqrt{\Delta_g})-(1-\chi)\cos (t\sqrt{\Delta_{\rr^{n+1}}})(1-\chi)).
\]
By a straightforward computation, for all $\chi\in
C_0^\infty(\rr^{n+1})$, $\tra(\chi \cos (t\sqrt{\Delta_{\rr^{n+1}}}))$
is a distribution supported at $t=0$ and its Fourier transform is a
polynomial of order $n$. Thus, using the definition~\eqref{bbtrace} of
the black-box trace, we easily see that there exists a polynomial
$P_n$ of degree $n$ such that for all $\varphi\in
C_0^\infty((0,\infty))$
\[
\frac{1}{2\pi}\cjg\hat{u},\varphi\cjd=\tra_{\rm bb}(\varphi(\sqrt{\Delta_g})-\varphi(\sqrt{\Delta}))+
\cjg P_n,\varphi\cjd
=- \int_0^\infty (\pl_zs(z)-P_n(z)) \varphi(z)dz . 
\]
It then suffices to study the asymptotics of $\int_0^z \hat{u}(w)dw$
as $z\to \infty$. In~\cite[Theorem~2]{SjZwJFA}, Sj\"ostrand--Zworski proved the
following Poisson formula extending~\cite{Me0} (see also \cite[Theorem~2]{Zw0} for a direct proof without using Lax--Phillips theory)
\[
u(t)=\sum_{\rho\in \mc{R}}e^{i\rho |t|}, \quad t\not =0.
\]
Moreover, $u(t)$ is classically conormal at $t=0$, this follows from
writing the wave propagator as a (classical) Fourier Integral Operator
(see e.g.~\cite[Sections~17.4 and~17.5]{Ho}): if $\theta \in
C_0^\infty(-t_0+\delta,t_0-\delta)$ for $\delta>0$ small enough, $\theta(t)\in [0,1]$ 
and $\theta(t)=1$ in $[-t_0+2\delta,t_0-2\delta]$, then
\begin{equation}
  \label{smalltime}
\mc{F}(u(t)\theta(t))(z)\sim \sum_{k=0}^{\infty}a_k^{\pm }z^{n-k}\,\, \textrm{ as }z\to \pm \infty
\end{equation}
for some coefficients $a_k^\pm$ (which are integrals of local
invariants in the metric $g$). Here $\mathcal F$ denotes the Fourier
transform and $t_0$ is the injectivity radius of $M$.  Notice that, by
\eqref{upperbound}, if $L$ is large enough, $t^L\sum_{\rho\in
\mc{R}}e^{i\rho |t|}$ is a well defined distribution on $\rr$ and it
equals $t^Lu(t)$ (in fact $L$ can be taken to be $n$, see~\cite[Theorem~7]{Zw0}).

Let $c>0$ and define the region (see Figure~\ref{f:1})
%
%
\begin{figure}
\includegraphics{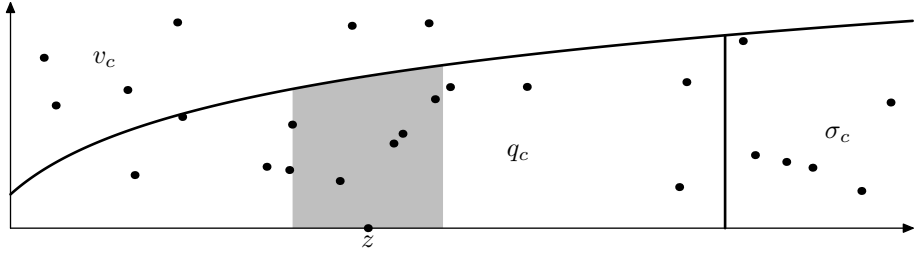}
\caption{The sets of resonances used in the proof of Theorem~\ref{usingresonance}.
We mark the regions corresponding to the sums $v_c$, $q_c$, and $\sigma_c$.
$\mathcal R_c$ is the region below the curve, while the shaded region
corresponds to $\mathcal N_{c,\log}(z)$.}
\label{f:1}
\end{figure}
%
%
\[
\mc{R}_c:=\{\rho\in \mc{R};\, {\rm Im}(\rho)\leq 1+ c\log|{\rm Re}(\rho)|\}.
\]
We decompose $u(t)(1-\theta(t))$ as 
\[
u(t)(1-\theta(t))=u_c(t)+v_c(t),\quad  \textrm{with }u_c(t)= \sum_{\rho\in \mc{R}_c}e^{i\rho|t|}(1-\theta(t)).
\]
Both terms are tempered distributions on $\rr$. 
From \eqref{upperbound}, one has  $v_c(t)\in L^p(\rr)$ for all $p>1$ if $ct_0>n+1$ and $\delta$ small enough: 
indeed
\[
|v_c(t)|\leq \sum_{\rho\in \mc{R}\setminus \mc{R}_c}|1-\theta(t)|e^{-{\rm Im}(\rho)|t|}
\leq C(1-\theta(t))\int_1^\infty R^{-c|t|+n} dR\leq \frac{C(1-\theta(t))}{c|t|-n-1}.
\]
This inequality also implies $|v_c(t)/t| \leq C(1+t^2)^{-1}$ and
therefore we get $\mc{F}(v_c(t)/t)\in L^\infty(\rr)\cap C^0(\rr)$.  We
thus obtain
\begin{equation}
  \label{vcint} 
\int_{0}^z\hat{v}_c(w)dw=i \mc{F}\Big(\frac{v_c(t)}{t}\Big)(z)-i\int \frac{v_c(t)}{t}dt \in L^\infty(\rr).
\end{equation}
In particular, the left-hand side of~\eqref{vcint} is $\mathcal O(z^\alpha)$ as $z\to \infty$.

It remains to consider the Fourier transform of $u_c$.  Let $\psi\in
C_0^\infty((-2,2))$ be equal to $1$ on $(-1,1)$, then we are going to
show that
\begin{equation}
  \label{hatuN}
\begin{gathered}\hat{u}_c(z)=2\pi\sigma_c(z)-\sigma_c\star \hat{\theta}(z)+2\pi q_c(z)-q_c\star\hat{\theta}(z) ,\\
 \quad \textrm{ with }q_c(z):=
\sum_{\rho\in \mc{R}_c}\psi\Big(\frac{|\rho|}{z}\Big)\frac{2{\rm Im}(\rho)}{|\rho-z|^2}
\end{gathered}
\end{equation}
and $\sigma_c(z)$ a symbol, that is there exists $L'$ such that 
for all $k\in\nn$, there is $C_k>0$ such that for all $z\in\rr$ 
\[
|\pl_z^k\sigma_c(z)|\leq C_k(1+|z|)^{L'-k}.
\]
The proof follows closely the argument of Melrose \cite{Me1}.
We have as tempered distribution
\[
\mc{F}(e^{i\rho|t|}+e^{-i\bar{\rho}|t|})(z)=\frac{2{\rm Im}(\rho)}{|z-\rho|^2}+\frac{2{\rm Im}(-\bar{\rho})}{|z+\bar{\rho}|^2}
\] 
and since the resonances are symmetric with respect to the imaginary axis, we have 
\[
\mc{F}\Big(t^L\sum_{\rho\in \mc{R}_c}e^{i\rho|t|}\Big)(z)=\sum_{\rho \in \mc{R}_c}(i\pl_z)^L\Big(\frac{2{\rm Im}(\rho)}{|z-\rho|^2}\Big)
\]
and these make sense as tempered distribution on $\rr$ by \eqref{upperbound}. Then 
\[
(i\pl_z)^L(\hat{u}_c(z)-q_c\star(2\pi\delta_0-\hat{\theta})(z))=\sum_{\rho\in \mc{R}_c}(i\pl_z)^L
 \Big(\Big(1-\psi\Big(\frac{|\rho|}{\cdot}\Big)\Big)\frac{2{\rm Im}(\rho)}{|\cdot-\rho|^2}\Big)\star(2\pi\delta_0-\hat{\theta})(z)
\]
and one can easily see that 
\[
\sigma_{c,L}(z):=\sum_{\rho\in \mc{R}_c}(i\pl_z)^L
 \Big(\Big(1-\psi\Big(\frac{|\rho|}{z}\Big)\Big)\frac{2{\rm Im}(\rho)}{|z-\rho|^2}\Big)
\]
is a symbol (see~\cite{Me1} for details) if $L$ is large enough. Integrate $\sigma_{c,L}(z)$ $L$ times, this 
defines a classical symbol with an order increasing by $L$, that we denote $\sigma_c(z)$ and such that
$(i\pl_z)^L(\sigma_c(z))=\sigma_{c,L}(z)$, and we therefore deduce that  
\[
\hat{u}_c(z)-q_c\star(2\pi\delta_0-\hat{\theta})(z)-\sigma_c\star(2\pi\delta_0-\hat{\theta})(z)=P_L(z)
\]
where $P_L(z)$ is a polynomial of order $L$. Taking inverse Fourier transform of the left hand side 
give a distribution vanishing near $t=0$ and therefore $P_L(z)=0$ necessarily. This shows \eqref{hatuN}.
Since for $z$ large $2\pi\sigma_c(z)-\sigma_c\star \hat{\theta}(z)=\mc{O}(z^{-\infty})$, we also get 
\begin{equation}
  \label{sigmaN}
\hat{u}_c(z)=q_c\star(2\pi\delta_0-\hat{\theta})(z)+\mc{O}(z^{-\infty}).
\end{equation}
An easy computation yields
\begin{equation}
  \label{calcul=pi} 
\forall a\leq b,\,\,\, \int_{a}^{b}\frac{{\rm Im}(\rho)}{|z-\rho|^2}dz \leq \pi
\end{equation} 
and if ${\rm Re}(\rho)+1 < a< b$ or $a<b<{\rm Re}(\rho)-1$ 
\begin{equation}
  \label{calcul2} 
\int_{a}^{b}\frac{{\rm Im}(\rho)}{|z-\rho|^2}dz \leq \frac{(b-a){\rm Im}(\rho)}{(b-{\rm Re}(\rho))(a-{\rm Re}(\rho))}.
\end{equation}
Define $\mc{N}_{c,\log}(z):=\sharp\{\rho\in \mc{R}_c; |z-{\rm Re}(\rho)|\leq \log(z)\}$ and set $Q_c(z)=\int_{0}^zq_c(w)dw$. Then we obtain from \eqref{calcul=pi} and \eqref{calcul2} that for $w\in[0,1]$ and $z\gg 1$
\[
|Q_c(z)-Q_c(z-w)|\leq C\Big(\mc{N}_{c,\log}(z)+ 
\sum_{\substack{\rho\in\mc{R}_c, |\rho|\leq 2z\\
|{\rm Re}(\rho)-z|>\log(z)}}\frac{{\rm Im}(\rho)}{(z-{\rm Re}(\rho))^2}\Big).
\] 
Using the assumption $\mc{N}_{c,\log}(z)=\mc{O}(z^{\alpha})$, we easily deduce for $z\gg 1$ and $w\in[0,1]$
\[|Q_c(z)-Q_c(z-w)|=\mathcal O(|z|^{\alpha}).\]
Thus iterating and using the same argument for $z<0$, we get for all $z, t\in\rr$ 
\[|Q_c(z)-Q_c(z-w)|\leq C(1+|w|)(1+|z|+|w|)^{\alpha}.\]
We multiply by $\hat{\theta}(t)$ and integrate to deduce 
\[|2\pi Q_c(z)-Q_c\star\hat{\theta}(z)|\leq C(1+|z|)^{\alpha}.\] 
It follows by~\eqref{sigmaN} that
$$
\Big|\int_0^z\hat u_c(w)\,dw\Big|\leq C(1+|z|)^{\alpha}.
$$
Combining this with \eqref{vcint} and \eqref{smalltime}, this yields
that as $z\to \infty$
\[
\int_0^z \hat{u}(w)dw =\sum_{j=0}^{n+1} c_jz^{n+1-j}+\mc{O}(z^\alpha)
\]
for some $c_j\in \rr$. This achieves the proof.
\end{proof}
%
%
Note that the proof applies for obstacles since in that case the small
time asymptotics of the wave trace is proved by Ivrii~\cite{Iv}.

\section{Scattering phase asymptotics for hyperbolic quotients}
  \label{s:k-1-s}

\subsection{Convex co-compact groups} 

A \emph{convex co-compact group} $\Gamma$ of isometries of hyperbolic
space $\hh^{n+1}$ is a discrete group of hyperbolic transformations
(their fixed points on $\bbar{\mathbb{B}}=\{m\in \rr^{n+1}; |m|\leq
1\}$ are $2$ disjoint points on $\mathbb S^n=\pl\bbar{M}$) with a
compact convex core.  The limit set $\Lambda_\Gamma$ of the group and
the discontinuity set $\Omega_\Gamma$ are defined by
\begin{equation}
  \label{limitset}
 \Lambda_\Gamma:=\overline{\{\gamma(m)\in \hh^{n+1}; \gamma\in \Gamma\}}\cap \mathbb S^n\,, \quad \Omega_\Gamma:=\mathbb S^n\setminus \Lambda_\Gamma\end{equation}
where the closure is taken in the closed unit ball $\bbar{\mathbb{B}}$
and $m\in \hh^{n+1}$ is any point.  The group $\Gamma$ acts on the
convex hull of $\Lambda_\Gamma$ (with respect to hyperbolic geodesics)
and the convex core is the quotient space. The group $\Gamma$ acts
properly discontinuously on $\Omega_\Gamma$ as a group of conformal
transformation and for convex co-compact groups, the quotient
$\Gamma\backslash \Omega_\Gamma$ is compact. The manifold
$M=\Gamma\backslash \hh^{n+1}$ admits a natural smooth
compactification given by $\bbar{M}= \Gamma\backslash (\hh^{n+1}\cup
\Omega_\Gamma)$, where the smooth structure is induced by that of the
closed unit ball $\bbar{\mathbb{B}}$ compactifying $\hh^{n+1}$. The
hyperbolic metric on $M$ is asymptotically hyperbolic in the sense of
Mazzeo--Melrose~\cite{MaMe}.  Let $x$ be a non-negative function which
is a boundary defining function for $\plM=\Gamma\backslash
\Omega_\Gamma$. We say that $x$ is a \emph{geodesic boundary defining
function} if $|d\log x|=1$ near $\plM$, where the norm is taken with
respect to the hyperbolic metric on $M$. Such functions exist, this is
proved for instance by Graham~\cite[Lemma 2.1]{Gr}.

An important quantity is the Hausdorff dimension of $\Lambda_\Gamma$
\begin{equation}
  \label{delta}
\delta:=\dim_{\rm H}\Lambda_{\Gamma} < n
\end{equation} 
which in turn is, by Patterson~\cite{Pa} and Sullivan~\cite{Su}, the
exponent of convergence of Poincar\'e series
\begin{equation}
  \label{patsull}
m\in \hh^{n+1}, \quad  \sum_{\gamma\in \Gamma} e^{-s d(m,\gamma m)}<\infty \iff s>\delta.
\end{equation}
where here and below, $d(\cdot,\cdot)$ is the distance for the
hyperbolic metric on $\hh^{n+1}$.  Note that the series is locally
uniformly bounded in $m\in\hh^{n+1}$. The Hausdorff dimension of the
trapped set $K$ of the geodesic flow on $\Gamma\backslash\hh^{n+1}$ is
given by $2\delta+1$, see Sullivan~\cite{Su}.

\subsection{Krein's type spectral function}
\label{s:krein}

In~\cite[Theorem~1.3]{GuAJM}, the second author defined on even
dimensional convex co-compact hyperbolic manifolds
$M=\Gamma\backslash\hh^{n+1}$ a generalized Krein function $\xi(z)$,
which is a regularized trace of the spectral projector
$\Pi_{[n^2/4,n^2/4+z^2]}(\Delta)$, and replaces the spectral shift
function in this setting. As we will see, it is related to the scattering phase.
It is defined as follows
\begin{equation}
  \label{defofxi}
\xi(z)=\int_0^{z^2} \tra_R(d\Pi_u(\Delta)), \quad z>0
\end{equation}
where $d\Pi_u(\Delta)$ is the spectral measure of the Laplacian,
i.e. $\int_{0}^zd\Pi_u(\Delta)=\indic_{[\frac{n^2}{4},\frac{n^2}{4}+u]}(\Delta)$,
and the regularized trace is defined using Hadamard finite part: let
$x$ be a geodesic boundary function in $\bbar{M}$, then if $T$ is an
operator with smooth integral kernel $T(m,m')$ such that the
restriction to the diagonal $m\mapsto T(m,m)$ has an asymptotic
expansion at $\pl\bbar{M}$ in powers of $x(m)$, we set
\[
\tra_R(T)={\rm FP}_{\eps\to 0}\int_{x>\eps} T(m,m)\Vol_g(m)
\]
and ${\rm FP}$ stands for finite part (ie. the coefficient of $\eps^0$
in the expansion as $\eps\to 0$).  The renormalized trace should a
priori depend on the choice of $x$ but in fact
$\tra_R(d\Pi_u(\Delta))$ does not depend on $x$, see~\cite{GuAJM}.  It
is natural to use this definition in this setting as there is no model
operator to compare with, unlike for perturbations of Euclidean
space. The function $\xi'(z)$ has meromorphic extension to $z\in \cc$
and $\xi$ is a phase of a regularized determinant of the scattering
operator $S(n/2+iz)$, see~\cite[Theorem~1.2]{GuAJM}. The following
identity is proved in~\cite[Theorem~1.3]{GuAJM} when $n+1$ is even
($\Gamma(\la)$ is the Euler function)
\begin{equation}
  \label{eqfunc}
\frac{Z_{\Gamma}(\ndemi-iz)}{Z_{\Gamma}(\ndemi+iz)}=\exp\Big(-2i\pi\xi(z)
+\chi(M)\frac{2i\pi(-1)^{\frac{n+1}{2}}}{\Gamma (n+1)}\int_{0}^z
\frac{\Gamma(\ndemi+it)\Gamma(\ndemi-it)}{\Gamma(it)\Gamma(-it)}dt +ci\pi \Big)
\end{equation} 
where $c\in\nn$ is a constant given by the multiplicity of $n/2$ as a
resonance, $\chi(M)$ is the Euler characteristic, $Z_{\Gamma}(\la)$ is
the Selberg zeta function, defined for $\Re(\la)>\delta$ (here
$\delta$ is defined in \eqref{delta})
\[
Z_{\Gamma}(\la)=\exp\left(-\sum_{\gamma}\sum_{m=1}^{\infty}\frac{1}{m}\frac{e^{-\la ml(\gamma)}}{G_\gamma(m)}\right)
\]
where $\gamma$ runs over the set $\mc{L}$ of primitive closed geodesics of $X$,
$\ell(\gamma)$ is the length of $\gamma$ and $G_\gamma(m):=e^{-\ndemi
ml(\gamma)}|\det(1-P_\gamma^m)|^{\demi}$ if $P_\gamma$ is the
Poincar\'e linear map associated to the primitive periodic orbit
$\gamma$ of the geodesic flow on the unit tangent bundle. The Selberg
zeta function extends to $\cc$ and has only finitely many zeros in
$\{{\rm Re}(\la)\geq \ndemi\}$, all on the real line by
Patterson--Perry~\cite{PP}. The zeros of $Z_\Gamma(\la)$ in $\{{\rm
Re}(\la)<n/2\}$ are the poles of the meromorphic continuation of the
resolvent $R(\la)=(\Delta-\la(n-\la))^{-1}$ of the Laplacian from
$\{{\rm Re}(\la)>n/2\}$ to $\cc$ (as a map $C_0^\infty(M)\to
C^\infty(M)$).  Keeping the convention in Section~\ref{s:resonances} for resonances,
we say that $\rho\in \{{\rm Im}(\rho)\geq 0\}$ is a resonance if this is a
pole of the meromorphic continuation of $z\to (\Delta-n^2/4-z^2)^{-1}$
from the lower half plane to $\cc$.  The function with Gamma factors
can be written as
\[
\frac{\Gamma(\ndemi+iz)\Gamma(\ndemi-iz)}{\Gamma(iz)\Gamma(-iz)}=\Big(\frac{1}{4}+z^2\Big)\dots
\Big(\big(\ndemi-1\big)^2+z^2\Big)z(1+\mc{O}(z^{-\infty}))
\] 
as $z\to \infty$. Notice that this is an odd polynomial modulo $\mc{O}(z^{-\infty})$. Since
$Z_{\Gamma}(\bar{\la})=\bbar{Z_{\Gamma}(\la)}$, we deduce that for $z>0$,
\[
\xi(z)=\chi(M)F_{n+1}(z)+\frac{1}{\pi}{\rm Arg}(Z_{\Gamma}(n/2+iz))+\demi c
\]
where $2i\pi F_{n+1}(z)$ is the  explicit function in factor of $\chi(M)$ in \eqref{eqfunc} (this is a 
polynomial modulo $\mc{O}(z^{-\infty})$), 
and ${\rm Arg}(Z_{\Gamma}(n/2+iz))$ is defined by 
\begin{equation}
  \label{argument}
{\rm Arg}(Z_{\Gamma}(n/2+iz))=\int_0^z {\rm Re}\Big(\frac{Z_{\Gamma}'(\ndemi+it)}{Z_{\Gamma}(\ndemi+it)}\Big)dt.
\end{equation}
 
\noindent\textbf{The case $\delta<n/2$}.
The series defining $Z_\Gamma(\ndemi+iz)$ converges uniformly for
$z\in \rr$ and therefore one obtains directly an explicit formula (in
this case $c=0$ as there is no resonance at $n/2$)
\begin{equation}
  \label{whendelta<n/2}
 \xi(z)=\chi(M)\frac{(-1)^{\frac{n+1}{2}}}{\Gamma (n+1)}\int_{0}^z \frac{\Gamma(\ndemi+it)\Gamma(\ndemi-it)}{\Gamma(it)\Gamma(-it)}dt +\frac{1}{\pi}\sum_{\gamma}\sum_{m=1}^{\infty}\frac{e^{-\ndemi m\ell(\gamma)}\sin(z\ell(\gamma))}{mG_\gamma(m)}.\end{equation}
The last term in the right hand side is $\mc{O}(1)$ as $z\to \infty$,
and by Gauss--Bonnet, we notice that $\chi(M)$ can be written as a
constant times a regularized volume ${\rm Vol}_R(M)$ of $M$, defined
by ${\rm Vol}_R(M)={\rm FP}_{\eps\to 0}\int_{x>\eps} 1 \,\Vol_g(m)$
(cf~\cite[Appendix~A]{PP}), corresponding to the usual Weyl type
asymptotics for spectral shift functions. The classical dynamic only
appears at order $z^0$ (with oscillations) in the asymptotic.\\
 
\noindent\textbf{The case $\delta\geq n/2$}. 
We can not use the convergence of $Z_\Gamma(\ndemi+iz)$ anymore. We
know by Guillop\'e--Lin--Zworski~\cite{GLZ} that for Schottky groups
$\Gamma$ (in particular all convex co-compact groups in dimension
$n+1=2$), one has for any $\sigma_0$, that there exists
$C_{\sigma_0}>0$ such that for all $z>1$ and $\sigma>\sigma_0$
\begin{equation}
  \label{estimateGLZ}
  \log |Z_{\Gamma}(\sigma+iz)|\leq C_{\sigma_0}z^{\delta}.
\end{equation}
%
%
\begin{lem}
  \label{ZN}
There exists $N>0$ such that for all $z\in\rr$, ${\rm Re}(Z_\Gamma(N+iz))>1/2$. 
\end{lem}
\begin{proof}
From the definition of $Z_\Gamma(\la)$ for ${\rm Re}(\la)>\delta$, we have 
\[
{\rm Re}(Z_{\Gamma}(\la))=
\exp\left(-\sum_{m=1}^{\infty}\frac{1}{m}{\rm Re}(T_\la(m))\right)
\cos\left(\sum_{m=1}^{\infty}\frac{1}{m}{\rm Im}(T_\la(m))\right)
\]
where $T_\la(m)=\sum_{\gamma}e^{-\la m\ell(\gamma)}/G_\gamma(m)$.
Let $\la=N+iz$ and we shall take $N$ large. We claim that there is $N$ large enough and $C>0$ 
such that for all $m\in \nn,z\in \rr$  
\[
|T_{N+iz}(m)|\leq Ce^{-Nm\ell(\gamma_0)}.
\]
where $\gamma_0$ is the geodesic with the smallest length. This
follows directly  from a lower bound
$|G_{\gamma}(m)|\geq c_0$ for some $c_0>0$ independent of $m$ and
$\gamma$, and the estimate
\[
\sharp\{ \gamma \in \mc{L}; \ell(\gamma)\leq R\}=\mc{O}(e^{\delta R}), \quad R\to \infty.
\]
To get the lower bound of $|G_{\gamma}(m)|$, it suffices to write this
term explicitly: associated to the primitive geodesic $\gamma$ there
is a conjugacy class of hyperbolic isometries in $\Gamma$, and
$\gamma$ is conjugated by a hyperbolic isometry to the transformation
$(x,y)\mapsto e^{\ell(\gamma)}(O_\gamma(x),y),$ in the half space
$(x,y)\in \hh^{n+1}=\rr^n\times \rr_+$ where $O_\gamma \in
SO_n(\rr)$. Denoting by $\alpha_1(\gamma),\ldots,\alpha_n(\gamma)$ the
eigenvalues of $O_\gamma$, then
\[
G_\gamma(m)=\det \left(I-e^{-m\ell(\gamma)}O_\gamma^m \right)=
\prod_{i=1}^n \left(1-e^{-m\ell(\gamma)}\alpha_i(\gamma)^m \right).
\]
Therefore $|G_\gamma(m)|\geq (1-e^{-m\ell(\gamma_0)})^n\geq
(1-e^{-\ell(\gamma_0)})^n$ where $\ell(\gamma_0)>0$ is the length of
the shortest geodesic on $M$.

We conclude that for any $\eps>0$, there is $C>0$ such that for all
$z\in\rr$
\[
\sum_{m=1}^{\infty}\frac{1}{m}|T_{N+iz}(m)| \leq Ce^{-N(\ell(\gamma_0)-\eps)}
\]
and taking $N$ large enough and $\eps$ small enough, we deduce the
Lemma.
\end{proof}
%
%
One can obtain an approximation of the argument of
$Z_\Gamma(\ndemi+iz)$ in terms of its zeros, this follows from a
general result on holomophic functions:
%
%
\begin{lem}
\label{breitwigner}
Assume that $\Gamma$ is a Schottky group in dimension $n+1$ even with Hausdorff dimension of limit set $\delta$. 
Then for all $\sigma>0$ fixed, there is a constant $C>0$ depending on $\sigma$ 
such that for all $z>0$ large and all $\la$ in the disc
\[
D(n/2+iz,\sigma):=\{\la\in \cc ; |\la-n/2-iz|\leq \sigma\},
\] 
we have the identity
\begin{equation}
\label{z'/z}
\Big|\frac{Z'_\Gamma(\la)}{Z_{\Gamma}(\la)} -\sum_{s\in\mc{R}_{\sigma+1}}\frac{1}{\la-s}\Big|\leq Cz^\delta
\end{equation}
where $\mc{R}_{\sigma+1}$ is the set of zeros of $Z_{\Gamma}(\la)$ in $D(n/2+iz,\sigma+1)$.
\end{lem}
\begin{proof}
We apply Lemma~$\alpha$ of~\cite[Section~3.9]{Tit} with the
holomorphic function $Z_{\Gamma}(\la)$ in a disc $D(N+iz,2N)$ where
$N>0$ is a large fixed parameter (chosen independent of $z$).  Using
the bound $|Z_{\Gamma}(\la)|\leq e^{Cz^\delta}$ of~\cite{GLZ} for
$\la\in D(N+iz,8N)$ where $C$ depends on $N$, and the lower bound
$|Z_{\Gamma}(N+iz)|>1/2$ for all $z>0$ if $N$ is large enough by Lemma~\ref{ZN},
we obtain that there exists $C'$ depending on $C,N$ such
that for all $\la\in D(N+iz,2N)$
\[
\Big|\frac{Z_{\Gamma}'(\la)}{Z_{\Gamma}(\la)} -\sum_{s\in\mc{R}_N}\frac{1}{\la-s}\Big|\leq C'z^\delta
\]
where $\mc{R}_N$ is the set of zeros of $Z_{\Gamma}(\la)$ in the disc
$D(N+iz,4N)$. To reduce to a sum
over $\mc{R}_{\sigma+1}$, it thus suffices to use that the number of
zeros of $Z_{\Gamma}(\la)$ in $D(N+iz,4N)$ is bounded by
$C''z^{\alpha}$  for some
$C''$ depends only on $N$ (by \cite{GLZ} again), and $|\la-s|\geq 1$ if $|\la-n/2-iz|\leq
\sigma$ and $s\in \mc{R}_N\setminus \mc{R}_{\sigma+1}$. This achieves
the proof.
\end{proof}
%
%
We also have an estimate on the argument of $Z_\Gamma(n/2+iz)$ by
using another Lemma in Titschmarch's book~\cite[Section~9.4]{Tit}.
%
%
\begin{lem}\label{Tit}
Fix $N>\sigma_0>0$. Let $f(z)$ be a holomorphic function in $\{{\rm
Re}(z)>0\}$, with $f(z)\in \rr$ for $z\in \rr$ and there exists $K>0$
such that $|{\rm Re} (f(N+it))|\geq K$ for all $t\in\rr$. Suppose that
$|f(\sigma'+it')|\leq M_{\sigma,t}$ for all $1\leq t'\leq t$ and
$\sigma'\geq \sigma$ with $\sigma\geq \sigma_0$. Then if $T$ is not
the imaginary part of a zero of $f$, we have
\[ |{\rm Arg}(f(\sigma+iT))|\leq C(\log(M_{\sigma_0,T+2})-\log(K))+3\pi/2\]
for some $C>0$ depending only on $\sigma_0$ and $N$.
\end{lem} 
%
%
We apply this to the function $Z_{\Gamma}(z)$ by using Lemma~\ref{ZN}
and~\eqref{estimateGLZ}, we obtain directly
%
%
\begin{cor}
If $\Gamma$ is a convex co-compact Schottky group of orientation
preserving isometries of $\hh^{n+1}$ with $n+1$ is even, and
$Z_{\Gamma}(\la)$ the Selberg zeta function for $\Gamma$, then for
$z>1$ large
\[
{\rm Arg}(Z_{\Gamma}(n/2+iz))=\mc{O}(z^\delta)
\]
where $\delta>0$ is the Hausdorff dimension of the limit set of $\Gamma$.
\end{cor}
%
%
In fact an integrated estimate is also true, by using a lemma of Hadamard: 
%
%
\begin{lem}
Let $\Gamma$ be a convex co-compact Schottky group of orientation
preserving isometries of $\hh^{n+1}$ with $n+1$ is even, and
$Z_{\Gamma}(\la)$ the Selberg zeta function for $\Gamma$ and
$\delta>0$ the Hausdorff dimension of the limit set. Then for $T>1$
large
\[
\int_{1}^T {\rm Arg}(Z_{\Gamma}(n/2+iz))dt=\mc{O}(T^{\delta}).
\]
\end{lem}
\begin{proof}
Let $T>1$. We apply the Hadamard Lemma to the holomorphic function
$Z_{\Gamma}(n/2-iz)$ in the rectangle $z\in [1,T]+i[0,N]$, with no
zeros in that rectangle: we get
\[
\begin{gathered}
\int_{0}^{N} \log |Z_{\Gamma}(n/2+u-i)|du-\int_{0}^{N}\log|Z_{\Gamma}(n/2+u-iT)| du=\\
\int_{1}^T{\rm Arg}(Z_{\Gamma}(n/2-iz))dz-\int_{1}^T{\rm Arg}(Z_{\Gamma}(N-iz))dz.
\end{gathered}
\]
Now, $|{\rm Arg}(Z_{\Gamma}(N-iz))|$ is bounded independently of $z$
when $N$ is large enough, thus the estimate \eqref{estimateGLZ} gives
the result when $\delta\geq 1$. When $\delta<1$, we take $N\sim
c\log(T)$ with $c>4/\ell(\gamma_0)$ and use that $\log|Z_{\Gamma}(n/2+u-iz)|\leq Cz^{\delta}$ in $u\in[0,u_0]$. By what
we explained above, if $u_0$ is large enough (and we take it
independent of $T$), we also have for all $z\in \rr$ and $u>u_0$
\[ |\log(Z_{\Gamma}(n/2+u-iz))|\leq Ce^{-u\ell(\gamma_0)/2}\]
thus 
\[
\begin{split}
\int_{0}^{N} \log |Z_{\Gamma}(n/2+u-i)|du+|\log|Z_{\Gamma}(n/2+u-iT)\|du &\leq CT^\delta+ \int_{u_0}^{N}e^{-u\ell(\gamma_0)/2}du  \\
& \leq CT^{\delta}
\end{split}
\]
and 
\[
\int_{1}^T|{\rm Arg}(Z_{\Gamma}(N-iz))|dz \leq \int_{1}^T e^{-N\ell(\gamma_0)/2} dt \leq CT^{-1}
\]
\end{proof}
%
%
Gathering all these results, we obtain (we assume $\delta\geq n/2$ as otherwise one has a complete formula 
by~\eqref{whendelta<n/2}) 
%
%
\begin{theo}
  \label{resultn+1}
Let $\Gamma$ be a convex co-compact Schottky group of orientation
preserving isometries of $\hh^{n+1}$ with $n+1$ is even and assume
that $\delta$, the Hausdorff dimension of the limit set, is larger or
equal to $n/2$. Then the Krein function for
$X=\Gamma\backslash\hh^{n+1}$ satisfies the Breit--Wigner
approximation: for $\sigma>0$ fixed and all $z\in [T-\sigma,T+\sigma]$
with $T$ large,
\begin{equation}
  \label{xi'}
\pl_z\xi(z)=\chi(X)P_{n}(z)+\frac{1}{\pi}\sum_{\rho\in \mc{R}_{\sigma+1}}\frac{{\rm Im}(\rho)}{|z-\rho|^2}+\mc{O}(T^\delta)
\end{equation}
where $\mc{R}_{\sigma+1}$ is the set of resonances in the disc of radius $\sigma+1$ centered at $T$, and $P_{n}(z)$ is an odd 
polynomial of degree $n$, $\chi(X)$ the Euler characteristic of $X$.
Moreover the following asymptotics holds
\[
\xi(z)=\chi(X)P_{n+1}(z)+R(z) , \quad |R(z)|=\mc{O}(z^\delta), \quad \Big|\int_{0}^T R(z)dz\Big|= \mc{O}(T^{\delta})
\]
where $P_{n+1}(z)$ is an even polynomial of degree $n+1$. 
\end{theo}
%
%

\subsection{The case of surfaces}

A convex co-compact surface $M=\Gamma\backslash \hh^2$ decomposes into
a compact hyperbolic surface with totally geodesic boundary (the
convex core $N$) and a collection of $n_f$ ends which are funnels
$F_j$ isometric to
\[
[0,\infty)_t \x (\rr/\ell_j\zz)_\theta \textrm{ with metric }dt^2+\cosh(t)^2 d\theta^2.
\]
This can be seen as half of the hyperbolic cylinder $\cjg
\gamma\cjd\backslash \hh^2$ where $\cjg \gamma\cjd$ is the cyclic
elementary group generated by one single hyperbolic isometry with
translation length $\ell_j$. Guillop\'e--Zworski~\cite{GZ2} studied
the spectral shift function $s(z)$%
\footnote{In~\cite{GZ2}, it was denoted $\sigma(z)$ instead of $s(z)$},
related to the pair of operators $(\Delta_{M}, \Delta_F)$ where
$\Delta_F:=\oplus_{j=1}^{n_f}\Delta_{F_j}$ where $\Delta_{F_j}$ is the
Laplacian on $F_j$ with Dirichlet condition at the boundary of $F_j$
(given by $t=0$ in the coordinate above). They define
$s(z)=\frac{i}{2\pi}\log \det \mc{S}(1/2+iz)$ for $z\in \rr$, where
$\mc{S}(s)$ is the relative scattering operator, and they show that
$s'(z)$ is meromorphic in $z\in \cc$ with poles and zeros given
essentially by resonances of the Laplacians. We do not give details
and refer the interested reader to~\cite{GZ2}. In fact, this function
$s(z)$ is given by
\begin{equation}
  \label{sprime} 
s'(z)=\xi'(z)-\xi'_F(z)
\end{equation}
where $\xi$ is the spectral function of \eqref{defofxi} for $M$ and
$\xi_F=\sum_j\xi_{F_j}(z)$ where $\xi_{F_j}$ is defined the same way
as \eqref{defofxi} on the funnel $F_j$ (and for the spectral measure
of $\Delta_{F_j}$); see the proof of Proposition~4.4 in~\cite{BoJuPe}.
It is computed by \cite[(2.10)]{BoJuPe} that there is a constant $c$
such that
\[
e^{2\pi i \xi_{F_j}(z)+c}=\frac{Z_{F_j}(1/2+iz)}{Z_{F_j}(1/2-iz)}
\]
where $Z_{F_j}(s):=e^{-s\ell_j/4}\prod_{k\geq
0}(1-e^{-(s+2k+1)\ell_j})^2$ is the Selberg zeta function for the
funnel $F_j$.  This function converges absolutely for ${\rm Re}(s)>0$
and in particular on the line $1/2+i\rr$. From this convergence, we
directly see that $\xi'_{F_j}(z)=\mc{O}(1)$ and we also deduce from
Lemma \ref{Tit} (as we did before for $Z_\Gamma$) that
$\xi_{F_j}(z)=\mc{O}(1)$ for $z\in \rr$. Thus by combining with
Theorem~\ref{resultn+1} and the Gauss--Bonnet formula ${\rm
Vol}(K)=-2\pi\chi(M)=2\pi(2g-2+n_f)$ (where $g$ is the genus), we
obtain
%
%
\begin{theo}
\label{R(z)remainder}
The scattering phase defined by Guillop\'e--Zworski~\cite{GZ2} for
hyperbolic convex co-compact surfaces $M=\Gamma\backslash\hh^2$
satisfies the asymptotics as $z\to +\infty$
\[
s(z)=\frac{1}{4\pi} {\rm Vol}(K)z^2 +R(z), \quad |R(z)|=\mc{O}(z^\delta), \quad \Big|\int_{0}^zR(t)dt\Big|=O(z^\delta)
\] 
where $K$ is the convex core of $M$, $\delta$ is the Hausdorff
dimension of the limit set of $\Gamma$. Moreover the Breit--Wigner
formula holds: for all $z\in [T-\sigma,T+\sigma]$ with $T$ large and
$\sigma$ fixed
\[
\pl_zs(z)=\frac{1}{4\pi}\vol(K)z^2+\frac{1}{\pi}\sum_{\rho\in \mc{R}_{\sigma+1}}\frac{{\rm Im}(\rho)}{|z-\rho|^2}+\mc{O}(z^\delta)
\]
\end{theo}
%
%

In~\cite{GZ2}, an asymptotics with a remainder $\mc{O}(z)$ using the
method of Melrose~\cite{Me1} was proved.  Our remainder estimate is
thus always better than~\cite{GZ2} since $\delta<1$.

\subsection{Lower bounds for the remainder}

For arithmetic hyperbolic surfaces, one can use the argument of
Jakobson--Naud~\cite{JaNa} to get a polynomial Omega lower bound on
the argument of Selberg zeta function on the line ${\rm Re}(s)=1/2$:
%
%
\begin{prop}{\bf [Jakobson--Naud]}
Let $\Gamma$ be a convex co-compact subgroup of isometries of $\hh^2$,
which is an infinite index subgroup of an arithmetic group $\Gamma_0$
derived from a quaternion algebra. Assume that the Hausdorff dimension
$\delta$ of the limit set satisfies $\delta>3/4$, then for all
$\eps>0$, the argument ${\rm Arg}(Z_\Gamma(1/2+iz))$ of Selberg zeta
function for $z\in\rr$ is {\bf not} an $\mc{O}(z^{2\delta-3/2-\eps})$.
\end{prop}
\begin{proof}
Let $Z_\Gamma(\la)$ be the Selberg zeta function of
$\Gamma$. In~\cite[Proposition~3.4]{JaNa}, it is shown that there
exist some constants $A>0,C>0$ such that for all $T>0$ large and
$\alpha=2\log(T)-A$, then
\begin{equation}
  \label{jaknaud} 
\int_{T}^{3T}|S_{\alpha,t}|^2dt\geq \frac{CT^{4\delta-2}}{(\log T)^2} 
\end{equation}
where $S_{\alpha,t}$ is defined for $\alpha>0,t>0$ as follows: let
$\varphi_0\in C_0^\infty((-1,1))$ non-negative equal to $1$ in
$[-1/2,1/2]$, then set
$\varphi_{t,\alpha}(x)=e^{-itx}\varphi_0(x-\alpha)$ and
\[
\begin{split} 
S_{\alpha,t}:=& \sum_{k\geq 1}\sum_{\gamma }\frac{\ell(\gamma)}{2\sinh(k\ell(\gamma)/2)}e^{-itk\ell(\gamma)}
\varphi_0(k\ell(\gamma)-\alpha)\\
 = &\frac{1}{\pi}\cjg \pl_z{\rm Arg}(Z_\Gamma(\demi+iz)), \widehat{\varphi_{t,\alpha}}(z)\cjd=-\frac{1}{\pi}{\cjg \rm Arg}(Z_\Gamma(\demi+iz)), 
 \pl_z\widehat{\varphi_{t,\alpha}}(z)\cjd.
\end{split}
\]
Now, assume that ${\rm Arg}(Z_\Gamma(\demi+iz))=\mc{O}(z^{2\delta-3/2-\eps})$ for some
$\eps>0$, then since $|\pl_z\widehat{\varphi_{t,\alpha}}(z)|=\mc{O}((1+|t+z|)^{-N})$ for
all $N\in\nn$, we get for $t>0$ large
\[
|S_{\alpha,t}|\leq Ct^{2\delta-3/2-\eps}
\]
where $C$ does not depend on $\alpha$. This estimate contradicts
\eqref{jaknaud} for $T$ large enough, this ends the proof.
\end{proof}
%
%
Note that such groups exist, see~\cite[Section~4]{JaNa}.
By~\eqref{sprime} and \eqref{eqfunc}, we then deduce the
%
%
\begin{cor}
If $\Gamma$ is a convex co-compact infinite index subgroup of an
arithmetic group $\Gamma_0$ derived from a quaternion algebra, and if
the Hausdorff dimension $\delta$ of the limit set of $\Gamma$
satisfies $\delta>3/4$, then for all $\eps>0$ the remainder $R(z)$ in
the scattering phase asymptotics of Theorem \ref{R(z)remainder} is {\bf
not} an $\mc{O}(z^{2\delta-3/2-\eps})$.
\end{cor}
%
%
In general, without the arithmetic assumption, the argument of
Jakobson--Naud~\cite{JaNa} shows that for all $\eps>0$, if
$\delta>1/2$ then the remainder $R(z)$ is not an $\mc{O}((\log
z)^{\frac{\delta-1/2}{\delta}-\eps})$.


\noindent\textbf{Acknowledgements.}
C.G. is supported by ANR grant ANR-09-JCJC-0099-01. S.D. was partially
supported by NSF grant DMS-0654436. We thank Johannes Sj\"ostrand and
Maciej Zworski for answering our questions on the counting functions
for resonances, as well as Fr\'ed\'eric Naud and Vesselin Petkov for
useful discussions and references.

\def\arXiv#1{\href{http://arxiv.org/abs/#1}{arXiv:#1}}

\end{document}